\documentclass[12pt]{amsart}
\usepackage{amsmath,amsthm,amsfonts,amssymb,mathrsfs}
\usepackage{color}

\usepackage{tikz}

\usepackage{amssymb,cite}
\usepackage[colorlinks,plainpages,citecolor=magenta, linkcolor=blue, backref]{hyperref}

\usepackage{hyperref}
\date{\today}

\usepackage{hyperref}
\input{xy}
 \xyoption{all}
 \xyoption{arc}

\newtheorem{theorem}{Theorem}[section]

\newtheorem{proposition}[theorem]{Proposition}
\newtheorem{corollary}[theorem]{Corollary}

\newtheorem{lemma}[theorem]{Lemma}
\theoremstyle{definition}

\begin{document}

\title[On injective endomorphisms of the semigroup $\boldsymbol{B}_{\omega}^{\mathscr{F}^4}$]{On injective  endomorphisms of the semigroup $\boldsymbol{B}_{\omega}^{\mathscr{F}^4}$ with a four-element family $\mathscr{F}^4$ of inductive non-empty subsets of $\omega$}
\author{Oleg Gutik and Marko Serivka}
\address{Ivan Franko National University of Lviv, Universytetska 1, Lviv, 79000, Ukraine}
\email{oleg.gutik@lnu.edu.ua, ogutik@gmail.com, marko.serivka@lnu.edu.ua}

\keywords{Bicyclic monoid, inverse semigroup, bicyclic extension, endomorphism, semigroup of endomorphisms, inductive set.}

\subjclass[2020]{20M18, 20F29, 20M10.}

\begin{abstract}
We describe injective endomorphisms of the semigroup $\boldsymbol{B}_{\omega}^{\mathscr{F}^4}$ with a four-element family $\mathscr{F}^4$ of inductive non-empty subsets of $\omega$.  In particular we proved that every injective monoid endomorphism of the semigroup $\boldsymbol{B}_{\omega}^{\mathscr{F}^4}$  is the identity transformation. Also we describe all (not necessary monoid) injective endomorphism of $\boldsymbol{B}_{\omega}^{\mathscr{F}^4}$.
\end{abstract}

\maketitle


\section{\textbf{Introduction, motivation and main definitions}}

We shall follow the terminology of the monographs~\cite{Clifford-Preston-1961, Clifford-Preston-1967, Lawson=1998}. By $\omega$ we denote the set of all non-negative integers. If $\mathfrak{f}\colon X\to Y$ is a map, then by $(x)\mathfrak{f}$ and $(A)\mathfrak{f}$ we denote the image of $x\in X$ and $A\subseteq X$ under $\mathfrak{f}$, respectively.

\smallskip

Let $\mathscr{P}(\omega)$ be  the family of all subsets of $\omega$. For any $F\in\mathscr{P}(\omega)$ and any integer $n$ we put $n+F=\{n+k\colon k\in F\}$ if $F\neq\varnothing$ and $n+\varnothing=\varnothing$.
A subfamily $\mathscr{F}\subseteq\mathscr{P}(\omega)$ is called \emph{${\omega}$-closed} if $F_1\cap(-n+F_2)\in\mathscr{F}$ for all $n\in\omega$ and $F_1,F_2\in\mathscr{F}$. For any $a\in\omega$ we denote $[a)=\{x\in\omega\colon x\geqslant a\}$.

\smallskip

A subset $A$ of $\omega$ is said to be \emph{inductive}, if $i\in A$ implies $i+1\in A$. Obvious, that $\varnothing$ is an inductive subset of $\omega$.

For an arbitrary semigroup $S$ any homomorphism $\alpha\colon S\to S$ is called an \emph{endomorphism} of $S$. If the semigroup has the identity element $1_S$ then the endomorphism $\alpha$ of $S$ such that $(1_S)\alpha=1_S$ is said to be a \emph{monoid endomorphism} of $S$. A bijective endomorphism of $S$ is called an \emph{automorphism}.

\smallskip

A semigroup $S$ is called {\it inverse} if for any
element $x\in S$ there exists a unique $x^{-1}\in S$ such that
$xx^{-1}x=x$ and $x^{-1}xx^{-1}=x^{-1}$. The element $x^{-1}$ is
called the {\it inverse of} $x\in S$.  

\smallskip

If $S$ is a semigroup, then we shall denote the subset of all
idempotents in $S$ by $E(S)$.  The semigroup
operation on $S$ determines the following partial order $\preccurlyeq$
on $E(S)$:
\begin{center}
$e\preccurlyeq f$ if and only if $ef=fe=e$.
\end{center}
This order is
called the {\em natural partial order} on $E(S)$. 

\smallskip

If $S$ is an inverse semigroup then the semigroup operation on $S$ determines the following partial order $\preccurlyeq$
on $S$: $s\preccurlyeq t$ if and only if there exists $e\in E(S)$ such that $s=te$. This order is
called the {\em natural partial order} on $S$ \cite{Wagner-1952}.

\smallskip

The \emph{bicyclic monoid} ${\mathscr{C}}(p,q)$ or the \emph{bicyclic monoid} is the semigroup with the identity $1$ generated by two elements $p$ and $q$ subjected only to the condition $pq=1$. The semigroup operation on ${\mathscr{C}}(p,q)$ is determined as
follows:
\begin{equation*}
    q^kp^l\cdot q^mp^n=q^{k+m-\min\{l,m\}}p^{l+n-\min\{l,m\}}.
\end{equation*}
It is well known that the bicyclic monoid ${\mathscr{C}}(p,q)$ is a bisimple (and hence simple) combinatorial $E$-unitary inverse semigroup and every non-trivial congruence on ${\mathscr{C}}(p,q)$ is a group congruence \cite{Clifford-Preston-1961}.

\smallskip

On the set $\boldsymbol{B}_{\omega}=\omega\times\omega$ we define the semigroup operation ``$\cdot$'' in the following way
\begin{equation}\label{eq-1.1}
  (i_1,j_1)\cdot(i_2,j_2)=
  \left\{
    \begin{array}{ll}
      (i_1-j_1+i_2,j_2), & \hbox{if~} j_1\leqslant i_2;\\
      (i_1,j_1-i_2+j_2), & \hbox{if~} j_1\geqslant i_2.
    \end{array}
  \right.
\end{equation}
It is well known that the bicyclic monoid $\mathscr{C}(p,q)$ is isomorphic to the semigroup $\boldsymbol{B}_{\omega}$ by the mapping $\mathfrak{h}\colon \mathscr{C}(p,q)\to \boldsymbol{B}_{\omega}$, $q^kp^l\mapsto (k,l)$ (see: \cite[Section~1.12]{Clifford-Preston-1961} or \cite[Exercise IV.1.11$(ii)$]{Petrich-1984}).

\smallskip

Next we shall describe the construction which is introduced in \cite{Gutik-Mykhalenych=2020}.

\smallskip

Let $\mathscr{F}$ be an ${\omega}$-closed subfamily of $\mathscr{P}(\omega)$. On the set $\boldsymbol{B}_{\omega}\times\mathscr{F}$ we define the semigroup operation ``$\cdot$'' in the following way
\begin{equation}\label{eq-1.2}
  (i_1,j_1,F_1)\cdot(i_2,j_2,F_2)=
  \left\{
    \begin{array}{ll}
      (i_1-j_1+i_2,j_2,(j_1-i_2+F_1)\cap F_2), & \hbox{if~} j_1\leqslant i_2;\\
      (i_1,j_1-i_2+j_2,F_1\cap (i_2-j_1+F_2)), & \hbox{if~} j_1\geqslant i_2.
    \end{array}
  \right.
\end{equation}
In \cite{Gutik-Mykhalenych=2020} is proved that if the family $\mathscr{F}\subseteq\mathscr{P}(\omega)$ is ${\omega}$-closed then $(\boldsymbol{B}_{\omega}\times\mathscr{F},\cdot)$ is a semigroup. Moreover, if an ${\omega}$-closed family  $\mathscr{F}\subseteq\mathscr{P}(\omega)$ contains the empty set $\varnothing$ then the set
$ 
  \boldsymbol{I}=\{(i,j,\varnothing)\colon i,j\in\omega\}
$ 
is an ideal of the semigroup $(\boldsymbol{B}_{\omega}\times\mathscr{F},\cdot)$. For any ${\omega}$-closed family $\mathscr{F}\subseteq\mathscr{P}(\omega)$ the following semigroup
\begin{equation*}
  \boldsymbol{B}_{\omega}^{\mathscr{F}}=
\left\{
  \begin{array}{ll}
    (\boldsymbol{B}_{\omega}\times\mathscr{F},\cdot)/\boldsymbol{I}, & \hbox{if~} \varnothing\in\mathscr{F};\\
    (\boldsymbol{B}_{\omega}\times\mathscr{F},\cdot), & \hbox{if~} \varnothing\notin\mathscr{F}
  \end{array}
\right.
\end{equation*}
is defined in \cite{Gutik-Mykhalenych=2020}. 
The semigroup $\boldsymbol{B}_{\omega}^{\mathscr{F}}$ is called the \emph{bicyclic extension generated by the family}  $\mathscr{F}$ of $\omega$-closed subsets of $\omega$  \cite{Gutik-Mykhalenych=2020}.

In \cite{Gutik-Mykhalenych=2020} it is proven that $\boldsymbol{B}_{\omega}^{\mathscr{F}}$ is a combinatorial inverse semigroup and Green's relations, the natural partial order on $\boldsymbol{B}_{\omega}^{\mathscr{F}}$ and its set of idempotents are described.
Also, in \cite{Gutik-Mykhalenych=2020} the criteria when the semigroup $\boldsymbol{B}_{\omega}^{\mathscr{F}}$ is simple, $0$-simple, bisimple, $0$-bisimple, or it has the identity, are given.
In particularly in \cite{Gutik-Mykhalenych=2020} it is proven that the semigroup $\boldsymbol{B}_{\omega}^{\mathscr{F}}$ is isomorphic to the semigrpoup of ${\omega}{\times}{\omega}$-matrix units if and only if $\mathscr{F}$ consists of a singleton set and the empty set, and $\boldsymbol{B}_{\omega}^{\mathscr{F}}$ is isomorphic to the bicyclic monoid if and only if $\mathscr{F}$ consists of a non-empty inductive subset of $\omega$.

\smallskip

Group congruences on the semigroup  $\boldsymbol{B}_{\omega}^{\mathscr{F}}$ and its homomorphic retracts  in the case when an ${\omega}$-closed family $\mathscr{F}$ consists of inductive non-empty subsets of $\omega$ are studied in \cite{Gutik-Mykhalenych=2021}. It is proven that a congruence $\mathfrak{C}$ on $\boldsymbol{B}_{\omega}^{\mathscr{F}}$ is a group congruence if and only if its restriction on a subsemigroup of $\boldsymbol{B}_{\omega}^{\mathscr{F}}$, which is isomorphic to the bicyclic semigroup, is not the identity relation. Also in \cite{Gutik-Mykhalenych=2021}, all non-trivial homomorphic retracts and isomorphisms  of the semigroup $\boldsymbol{B}_{\omega}^{\mathscr{F}}$ are described. In \cite{Gutik-Mykhalenych=2022} it is proven that an injective endomorphism $\varepsilon$ of the semigroup $\boldsymbol{B}_{\omega}^{\mathscr{F}}$ is the indentity transformation if and only if  $\varepsilon$ has three distinct fixed points, which is equivalent to existence non-idempotent element $(i,j,[p))\in\boldsymbol{B}_{\omega}^{\mathscr{F}}$ such that  $(i,j,[p))\varepsilon=(i,j,[p))$.

\smallskip

In \cite{Gutik-Lysetska=2021, Lysetska=2020} the algebraic structure of the semigroup $\boldsymbol{B}_{\omega}^{\mathscr{F}}$ is established in the case when ${\omega}$-closed family $\mathscr{F}$ consists of atomic subsets of ${\omega}$. The structure of the semigroup $\boldsymbol{B}_{\omega}^{\mathscr{F}_n}$, for the family  $\mathscr{F}_n$ which is generated by the initial interval $\{0,1,\ldots,n\}$ of $\omega$, is studied in \cite{Gutik-Popadiuk=2023}. The semigroup of endomorphisms of $\boldsymbol{B}_{\omega}^{\mathscr{F}_n}$ is described in \cite{Gutik-Popadiuk=2022, Popadiuk=2022}.

\smallskip

In \cite{Gutik-Prokhorenkova-Sekh=2021} it is proven that the semigroup $\mathrm{\mathbf{End}}(\boldsymbol{B}_{\omega})$ of all endomorphisms of the bicyclic semigroup $\boldsymbol{B}_{\omega}$ is isomorphic to the semidirect products $(\omega,+)\rtimes_\varphi(\omega,*)$, where $+$ and $*$ are the usual addition and the usual multiplication on the set of non-negative integers $\omega$, respectively.

\smallskip

Fix an arbitrary positive integer $n$.
Later we assume that $\mathscr{F}^n$ is a family of inductive non-empty subsets of $\omega$ which consists of $n$ sets.  By Proposition~1 of \cite{Gutik-Mykhalenych=2021} for any $\omega$-closed family $\mathscr{F}$ of inductive subsets in $\mathscr{P}(\omega)$ there exists an $\omega$-closed family $\mathscr{F}^*$ of inductive subsets in $\mathscr{P}(\omega)$ such that $[0)\in \mathscr{F}^*$ and the semigroups $\boldsymbol{B}_{\omega}^{\mathscr{F}}$ and $\boldsymbol{B}_{\omega}^{\mathscr{F}^*}$ are isomorphic. Hence without loss of generality we may assume that the family $\mathscr{F}$ contains the set $[0)$, i.e., $\mathscr{F}^n=\{[0),[1),\ldots,[n-1)\}$.

\smallskip

In the paper \cite{Gutik-Pozdniakova=2022}  injective endomorphisms of the semigroup $\boldsymbol{B}_{\omega}^{\mathscr{F}^2}$ with the two-elements family $\mathscr{F}^2$ of inductive non-empty subsets of $\omega$ are studied. Also, in \cite{Gutik-Pozdniakova=2022} the authors describe the elements of the semigroup $\boldsymbol{End}^1_*(\boldsymbol{B}_{\omega}^{\mathscr{F}^2})$ of all injective monoid endomorphisms of the monoid $\boldsymbol{B}_{\omega}^{\mathscr{F}^2}$, and show that Green's relations $\mathscr{R}$, $\mathscr{L}$, $\mathscr{H}$, $\mathscr{D}$, and $\mathscr{J}$  on $\boldsymbol{End}^1_*(\boldsymbol{B}_{\omega}^{\mathscr{F}^2})$ coincide with the relation of equality. In \cite{Gutik-Pozdniakova=2023a, Gutik-Pozdniakova=2023b} the semigroup $\boldsymbol{End}^1(\boldsymbol{B}_{\omega}^{\mathscr{F}^2})$ of all  monoid endomorphisms of the monoid $\boldsymbol{B}_{\omega}^{\mathscr{F}^2}$ is studied.

In \cite{Gutik-Serivka=2025} we study the semigroup $\overline{\boldsymbol{End}}(\boldsymbol{B}_{\omega}^{\mathcal{F}^2})$ of all endomorphisms of the bicyclic extension $\boldsymbol{B}_{\omega}^{\mathcal{F}^2}$ with the two-element family $\mathcal{F}^2$ of inductive non-empty subsets of $\omega$. The submonoid  $\left\langle\varpi\right\rangle^1$ of $\overline{\boldsymbol{End}}(\boldsymbol{B}_{\omega}^{\mathcal{F}^2})$ with the property that every element of the semigroup $\overline{\boldsymbol{End}}(\boldsymbol{B}_{\omega}^{\mathcal{F}^2})$ has the unique representation as the product of the monoid endomorphism of $\boldsymbol{B}_{\omega}^{\mathcal{F}^2}$ and the element of $\left\langle\varpi\right\rangle^1$ is constructed.

\smallskip

In \cite{Gutik-Serivka=2023} we describe injective monoid endomorphisms of the semigroup $\boldsymbol{B}_{\omega}^{\mathscr{F}^3}$. Here we shown that for
every injective monoid endomorphism $\varepsilon$ of $\boldsymbol{B}_{\omega}^{\mathscr{F}^3}$ there exists a positive integer $k$ such that $\varepsilon=\alpha_{[k]}$ where the mapping $\alpha_{[k]}\colon \boldsymbol{B}_{\omega}^{\mathscr{F}^3}\to\boldsymbol{B}_{\omega}^{\mathscr{F}^3}$ is defined by the formula
\begin{equation*}
  (i,j,[p))\alpha_{[k]}=
  \left\{
    \begin{array}{ll}
      (ki,kj,[p)), & \hbox{if~} p\in\{0,1\};\\
      (k(i+1)-1,k(j+1)-1,[2)), & \hbox{if~} p=2,
    \end{array}
  \right.
\end{equation*}
for all $i,j\in\omega$ (see Theorem~5 of \cite{Gutik-Serivka=2023}). Also we prove that the monoid $\boldsymbol{End}_{\textsf{inj}}^1(\boldsymbol{B}_{\omega}^{\mathscr{F}^3})$ of all injective monoid endomorphisms of the semigroup $\boldsymbol{B}_{\omega}^{\mathscr{F}^3}$ is isomorphic to the multiplicative semigroup of positive integers.

In this paper we describe injective monoid endomorphisms of the semigroup $\boldsymbol{B}_{\omega}^{\mathscr{F}^4}$. In particular we prove that every injective monoid endomorphism of the semigroup $\boldsymbol{B}_{\omega}^{\mathscr{F}^4}$  is the identity transformation. Also we describe all (not necessary monoid) injective endomorphism of $\boldsymbol{B}_{\omega}^{\mathscr{F}^4}$.

\section{\textbf{On properties of injective monoid endomorphisms of the semigroup $\boldsymbol{B}_{\omega}^{\mathscr{F}^4}$}}\label{section-2}

For any $l,m,n\in\{0,1,2,3\}$, $l<m<n$, we denote
\begin{align*}
  \boldsymbol{B}_{\omega}^{\mathscr{F}^4_l}      &=\{(i,j,[l))\colon i,j\in\omega\}; \\
  \boldsymbol{B}_{\omega}^{\mathscr{F}^4_{l,m}}  &=\boldsymbol{B}_{\omega}^{\mathscr{F}^4_{l}}\cup\boldsymbol{B}_{\omega}^{\mathscr{F}^4_{m}}; \\
  \boldsymbol{B}_{\omega}^{\mathscr{F}^4_{l,m,n}}&=\boldsymbol{B}_{\omega}^{\mathscr{F}^4_{l}}\cup\boldsymbol{B}_{\omega}^{\mathscr{F}^4_{m}}\cup \boldsymbol{B}_{\omega}^{\mathscr{F}^4_{n}}.
\end{align*}

\begin{proposition}\label{proposition-2.1}
Let $\varepsilon$ be an injective monoid endomorphism of $\boldsymbol{B}_{\omega}^{\mathscr{F}^4}$. If $(\boldsymbol{B}_{\omega}^{\mathscr{F}^4_{0,1,2}})\varepsilon\subseteq \boldsymbol{B}_{\omega}^{\mathscr{F}^4_{0,1,2}}$, then $\varepsilon$ is the identity transformation.
\end{proposition}

\begin{proof}
Since $(\boldsymbol{B}_{\omega}^{\mathscr{F}^4_{0,1,2}})\varepsilon\subseteq \boldsymbol{B}_{\omega}^{\mathscr{F}^4_{0,1,2}}$, the restriction $\varepsilon{\upharpoonright}_{\boldsymbol{B}_{\omega}^{\mathscr{F}^4_{0,1,2}}}\colon \boldsymbol{B}_{\omega}^{\mathscr{F}^4_{0,1,2}}\to \boldsymbol{B}_{\omega}^{\mathscr{F}^4_{0,1,2}}$ of the endomorphism $\varepsilon$ is an injective monoid endomorphism. Hence by Theorem~5 of \cite{Gutik-Serivka=2023} it is determined by the formula
\begin{equation*}
  (i,j,[p))\varepsilon{\upharpoonright}_{\boldsymbol{B}_{\omega}^{\mathscr{F}^4_{0,1,2}}}=(i,j,[p))\alpha_{[k]}=
  \left\{
    \begin{array}{ll}
      (ki,kj,[p)), & \hbox{if~} p\in\{0,1\};\\
      (k(i+1)-1,k(j+1)-1,[2)), & \hbox{if~} p=2,
    \end{array}
  \right.
\end{equation*}
for some positive integer $k$ and any $i,j\in\omega$, because by Proposition~1 of \cite{Gutik-Mykhalenych=2021} the semigroup  $\boldsymbol{B}_{\omega}^{\mathscr{F}^{3}}$ and the subsemigroup $\boldsymbol{B}_{\omega}^{\mathscr{F}^4_{0,1,2}}$ of $\boldsymbol{B}_{\omega}^{\mathscr{F}^4}$ are isomorphic. This implies that
\begin{equation*}
  (0,0,[p))\varepsilon{\upharpoonright}_{\boldsymbol{B}_{\omega}^{\mathscr{F}^4_{0,1,2}}}=(0,0,[p))\varepsilon=(0,0,[p))
\end{equation*}
for $p\in\{0,1\}$.

Suppose that $(\boldsymbol{B}_{\omega}^{\mathscr{F}^4_{3}})\varepsilon\subseteq \boldsymbol{B}_{\omega}^{\mathscr{F}^4_{1,2,3}}$. Since $(0,0,[1))\varepsilon=(0,0,[1))$ we conclude that the restriction $\varepsilon{\upharpoonright}_{\boldsymbol{B}_{\omega}^{\mathscr{F}^4_{1,2,3}}}\colon \boldsymbol{B}_{\omega}^{\mathscr{F}^4_{1,2,3}}\to \boldsymbol{B}_{\omega}^{\mathscr{F}^4_{1,2,3}}$ of the endomorphism $\varepsilon$ is an injective monoid endomorphism of $\boldsymbol{B}_{\omega}^{\mathscr{F}^4_{1,2,3}}$, and hence by Theorem~5 of \cite{Gutik-Serivka=2023} it is determined by the formula
\begin{equation*}
  (i,j,[p))\varepsilon{\upharpoonright}_{\boldsymbol{B}_{\omega}^{\mathscr{F}^4_{1,2,3}}}=
  \left\{
    \begin{array}{ll}
      (k_1i,k_1j,[p)), & \hbox{if~} p\in\{1,2\};\\
      (k_1(i+1)-1,k_1(j+1)-1,[2)), & \hbox{if~} p=3,
    \end{array}
  \right.
\end{equation*}
for some positive integer $k_1$ and any $i,j\in\omega$. This and the above part of the proof imply that
\begin{equation*}
  (0,0,[p))\varepsilon{\upharpoonright}_{\boldsymbol{B}_{\omega}^{\mathscr{F}^4_{0,1,2}}}=(0,0,[p))\varepsilon=(0,0,[p))
\end{equation*}
for any $p\in\{0,1,2\}$. Therefore the endomorphism $\varepsilon$ has distinct three fixed points and hence by Theorem~2 of \cite{Gutik-Mykhalenych=2022} $\varepsilon$ is the identity transformation.

Suppose that $(\boldsymbol{B}_{\omega}^{\mathscr{F}^4_{3}})\varepsilon\subseteq \boldsymbol{B}_{\omega}^{\mathscr{F}^4_{0}}$. By the previous part of the proof we have that $(0,0,[2))\varepsilon{=}({k{-}1},{k{-}1},[2))$ for some positive integer $k$. Proposition~1.4.21 of \cite{Lawson=1998} state that a homomorhism of inverse semigroups  preserves the natural partial order and idempotents, and hence by the description of the natural partial order on semilattice of idempotents $E(\boldsymbol{B}_{\omega}^{\mathscr{F}^4})$ of the semigroup $\boldsymbol{B}_{\omega}^{\mathscr{F}^4}$ (see Proposition~3 in \cite{Gutik-Mykhalenych=2021}) by Lemma~2 of \cite{Gutik-Mykhalenych=2020} we get that $(0,0,[3))\varepsilon=(s,s,[0))$ for some positive integer $s\geqslant k+1$. Then we get that
\begin{equation*}
  ((1,1,[0))\cdot (0,0,[3)))\varepsilon=(1,1,[2))\varepsilon=(2k-1,2k-1,[2))
\end{equation*}
and
\begin{equation*}
  (1,1,[0))\varepsilon\cdot (0,0,[3))\varepsilon=(k,k,[0))\cdot(s,s,[0))=(s,s,[0)).
\end{equation*}
This contradicts that $\varepsilon$ is an endomorphism of the semigroup $\boldsymbol{B}_{\omega}^{\mathscr{F}^4}$. The obtained contradiction implies that $(\boldsymbol{B}_{\omega}^{\mathscr{F}^4_{3}})\varepsilon\nsubseteq \boldsymbol{B}_{\omega}^{\mathscr{F}^4_{0}}$. By Proposition~4 of \cite{Gutik-Mykhalenych=2021} we obtain that $(\boldsymbol{B}_{\omega}^{\mathscr{F}^4_{3}})\subseteq \boldsymbol{B}_{\omega}^{\mathscr{F}^4_{1,2,3}}$. Then by above part of the proof we have that $\varepsilon$ is the identity transformation of $\boldsymbol{B}_{\omega}^{\mathscr{F}^4}$.
\end{proof}

\begin{proposition}\label{proposition-2.2}
Let $\varepsilon$ be an injective monoid endomorphism of $\boldsymbol{B}_{\omega}^{\mathscr{F}^4}$. If $(\boldsymbol{B}_{\omega}^{\mathscr{F}^4_{1,2,3}})\varepsilon\subseteq \boldsymbol{B}_{\omega}^{\mathscr{F}^4_{1,2,3}}$, then $\varepsilon$ is the identity transformation.
\end{proposition}

\begin{proof}
The assumption of the proposition implies that the restriction $\varepsilon{\upharpoonright}_{\boldsymbol{B}_{\omega}^{\mathscr{F}^4_{1,2,3}}}\colon \boldsymbol{B}_{\omega}^{\mathscr{F}^4_{1,2,3}}\to \boldsymbol{B}_{\omega}^{\mathscr{F}^4_{1,2,3}}$ of the endomorphism $\varepsilon$ is an injective  endomorphism of the subsemigroup $\boldsymbol{B}_{\omega}^{\mathscr{F}^4_{1,2,3}}$ of the monoid $\boldsymbol{B}_{\omega}^{\mathscr{F}^4}$. By Theorem~1 of \cite{Gutik-Serivka=2026} we have that only one of the following conditions holds
\begin{equation*}
  (\boldsymbol{B}_{\omega}^{\mathscr{F}^4_{1}})\varepsilon{\upharpoonright}_{\boldsymbol{B}_{\omega}^{\mathscr{F}^4_{1,2,3}}}= (\boldsymbol{B}_{\omega}^{\mathscr{F}^4_{1}})\varepsilon\subseteq \boldsymbol{B}_{\omega}^{\mathscr{F}^4_{1}} \qquad \hbox{or} \qquad
  (\boldsymbol{B}_{\omega}^{\mathscr{F}^4_{1}})\varepsilon{\upharpoonright}_{\boldsymbol{B}_{\omega}^{\mathscr{F}^4_{1,2,3}}}= (\boldsymbol{B}_{\omega}^{\mathscr{F}^4_{1}})\varepsilon\subseteq \boldsymbol{B}_{\omega}^{\mathscr{F}^4_{3}}.
\end{equation*}

Suppose that the inclusion $(\boldsymbol{B}_{\omega}^{\mathscr{F}^4_{1}})\varepsilon{\upharpoonright}_{\boldsymbol{B}_{\omega}^{\mathscr{F}^4_{1,2,3}}}= (\boldsymbol{B}_{\omega}^{\mathscr{F}^4_{1}})\varepsilon\subseteq \boldsymbol{B}_{\omega}^{\mathscr{F}^4_{1}}$ holds. Proposition~1.4.21 of \cite{Lawson=1998} states that a homomorphism of semigroups  preserves idempotents and hence we have that $(0,0,[1))\varepsilon=(k,k,[1))$ for some non-negative integer $k$. By Proposition~1 of \cite{Gutik-Mykhalenych=2021} the semigroups $\boldsymbol{B}_{\omega}^{\mathscr{F}^{3}}$ and $\boldsymbol{B}_{\omega}^{\mathscr{F}^4_{1,2,3}}$ are isomorphic by the mapping $\mathfrak{I}\colon \boldsymbol{B}_{\omega}^{\mathscr{F}^{3}}\to \boldsymbol{B}_{\omega}^{\mathscr{F}^4_{1,2,3}}$, $(i,j[p))\mapsto (i,j,[p+1))$ for any $i,j\in\omega$ and $p\in\{0,1,2\}$. Then the mapping $\varepsilon^{\prime}=\mathfrak{I}\circ \varepsilon{\upharpoonright}_{\boldsymbol{B}_{\omega}^{\mathscr{F}^4_{1,2,3}}} \circ\mathfrak{I}^{-1}$ is an injective endomorphism of the semigroup $\boldsymbol{B}_{\omega}^{\mathscr{F}^{3}}$ such that $(0,0,[0))\varepsilon^{\prime}= (k,k,[0))$. By Theorem~1 of \cite{Gutik-Serivka=2026} we obtain that $\varepsilon^{\prime}=\lambda^k$, where the map $\lambda\colon \boldsymbol{B}_{\omega}^{\mathscr{F}^3}\to \boldsymbol{B}_{\omega}^{\mathscr{F}^3}$ is defined by the formula
  \begin{equation}\label{eq-2.1}
  (i,j,[p))\lambda=(i+1,j+1,[p)), \quad  i,j\in\omega, \; p\in\{0,1,2\}.
  \end{equation}
Since by Proposition~3 of \cite{Gutik-Mykhalenych=2020} the subsemigroup $\boldsymbol{B}_{\omega}^{\mathscr{F}^4_{p}}$ of $\boldsymbol{B}_{\omega}^{\mathscr{F}^{4}}$ is isomorphic to the bicyclic semigroup, Proposition~4 of \cite{Gutik-Mykhalenych=2021} implies that
\begin{equation*}
  (\boldsymbol{B}_{\omega}^{\mathscr{F}^4_{p}})\varepsilon{\upharpoonright}_{\boldsymbol{B}_{\omega}^{\mathscr{F}^4_{1,2,3}}}= (\boldsymbol{B}_{\omega}^{\mathscr{F}^4_{p}})\varepsilon\subseteq \boldsymbol{B}_{\omega}^{\mathscr{F}^4_{p}},
\end{equation*}
for any $p\in\{1,2,3\}$. Then $(\boldsymbol{B}_{\omega}^{\mathscr{F}^4_{0,1,2}})\varepsilon\subseteq \boldsymbol{B}_{\omega}^{\mathscr{F}^4_{0,1,2}}$ and by Proposition~\ref{proposition-2.1} the endomorphism $\varepsilon$ is the identity transformation.

Suppose that the inclusion $(\boldsymbol{B}_{\omega}^{\mathscr{F}^4_{1}})\varepsilon{\upharpoonright}_{\boldsymbol{B}_{\omega}^{\mathscr{F}^4_{1,2,3}}}\subseteq \boldsymbol{B}_{\omega}^{\mathscr{F}^4_{3}}$ holds. By Proposition~1.4.21 of \cite{Lawson=1998} a homomorphic image of an idempotent is an idempotent, as well, and hence by Lemma~2 of \cite{Gutik-Mykhalenych=2020} there exists a non-negative integer $m$ such that $(0,0,[1))\varepsilon=(m,m,[3))$. By Proposition~1 of \cite{Gutik-Mykhalenych=2021} the subsemigroups $\boldsymbol{B}_{\omega}^{\mathscr{F}^{3}}$ and $\boldsymbol{B}_{\omega}^{\mathscr{F}^4_{1,2,3}}$ are isomorphic by the mapping $\mathfrak{I}\colon \boldsymbol{B}_{\omega}^{\mathscr{F}^{3}}\to \boldsymbol{B}_{\omega}^{\mathscr{F}^4_{1,2,3}}$, $(i,j[p))\mapsto (i,j,[p+1))$ for any $i,j\in\omega$ and $p\in\{0,1,2\}$. Then the mapping $\varepsilon^{\prime}=\mathfrak{I}\circ \varepsilon{\upharpoonright}_{\boldsymbol{B}_{\omega}^{\mathscr{F}^4_{1,2,3}}} \circ\mathfrak{I}^{-1}$ is an injective endomorphism of the semigroup $\boldsymbol{B}_{\omega}^{\mathscr{F}^{3}}$ such that $(0,0,[0))\varepsilon^{\prime}= (m,m,[2))$. Again we use Theorem~1 of \cite{Gutik-Serivka=2026} and get that $\varepsilon^{\prime}=\alpha_{[k]}\circ\lambda^m\circ\varpi_3$ for some positive integer $k$, where the endomorphisms $\lambda$, $\alpha_{[k]}$, and $\varpi_3$ are defined by the formulae \eqref{eq-2.1}, \eqref{eq-2.2}, and \eqref{eq-2.3},
\begin{equation}\label{eq-2.2}
  (i,j,[p))\alpha_{[k]}=
  \left\{
    \begin{array}{ll}
      (ki,kj,[p)), & \hbox{if~} p\in\{0,1\};\\
      (k(i+1)-1,k(j+1)-1,[2)), & \hbox{if~} p=2,
    \end{array}
  \right.
\end{equation}
 \begin{equation}\label{eq-2.3}
  (i,j,[p))\varpi_3=(i+p,j+p,[2-p)),
  \end{equation}
$i,j\in\omega$, $p\in\{0,1,2\}$, respectively.  Then we obtain that
\begin{align*}
  (i,j,[0))\varepsilon^{\prime} &=(i,j,[0))(\alpha_{[k]}\circ\lambda^m\circ\varpi_3)=\\
  &=(ki,kj,[0))(\lambda^m\circ\varpi_3)= \\
  &=(ki+m,kj+m,[0))\varpi_3=\\
  &=(ki+m,kj+m,[2)),
\end{align*}
\begin{align*}
  (i,j,[1))\varepsilon^{\prime} &=(i,j,[1))(\alpha_{[k]}\circ\lambda^m\circ\varpi_3)=\\
  &=(ki,kj,[1))(\lambda^m\circ\varpi_3)= \\
  &=(ki+m,kj+m,[1))\varpi_3=\\
  &=(ki+m,kj+m,[1)),
\end{align*}
\begin{align*}
  (i,j,[2))\varepsilon^{\prime} &=(i,j,[2))(\alpha_{[k]}\circ\lambda^m\circ\varpi_3)=\\
  &=(k(i+1)-1,k(j+1)-1,[2))(\lambda^m\circ\varpi_3)= \\
  &=(k(i+1)+m-1,k(j+1)+m-1,[2))\varpi_3=\\
  &=(k(i+1)+m-1,k(j+1)+m-1,[0)),
\end{align*}
for all $i,j\in\omega$.

Since $\varepsilon{\upharpoonright}_{\boldsymbol{B}_{\omega}^{\mathscr{F}^4_{1,2,3}}}=\mathfrak{I}^{-1}\circ \varepsilon^{\prime} \circ\mathfrak{I}$ we get that
\begin{align}\label{eq-2.4}
  (i,j,[1))\varepsilon{\upharpoonright}_{\boldsymbol{B}_{\omega}^{\mathscr{F}^4_{1,2,3}}}&= (i,j,[1))(\mathfrak{I}^{-1}\circ \varepsilon^{\prime} \circ\mathfrak{I})=\nonumber \\
  &=(i,j,[0))(\varepsilon^{\prime} \circ\mathfrak{I})=\nonumber\\
  &=(ki+m,kj+m,[2))\mathfrak{I}=\\
  &=(ki+m,kj+m,[3)),\nonumber
\end{align}
\begin{align}\label{eq-2.5}
  (i,j,[2))\varepsilon{\upharpoonright}_{\boldsymbol{B}_{\omega}^{\mathscr{F}^4_{1,2,3}}}&= (i,j,[2))(\mathfrak{I}^{-1}\circ \varepsilon^{\prime} \circ\mathfrak{I})=\nonumber \\
  &=(i,j,[1))(\varepsilon^{\prime} \circ\mathfrak{I})=\nonumber\\
  &=(ki+m,kj+m,[1))\mathfrak{I}=\\
  &=(ki+m,kj+m,[2)),\nonumber
\end{align}
\begin{align}\label{eq-2.6}
  (i,j,[3))\varepsilon{\upharpoonright}_{\boldsymbol{B}_{\omega}^{\mathscr{F}^4_{1,2,3}}}&= (i,j,[3))(\mathfrak{I}^{-1}\circ \varepsilon^{\prime} \circ\mathfrak{I})=\nonumber \\
  &=(i,j,[2))(\varepsilon^{\prime} \circ\mathfrak{I})=\nonumber\\
  &=(k(i+1)+m-1,k(j+1)+m-1,[0))\mathfrak{I}=\\
  &=(k(i+1)+m-1,k(j+1)+m-1,[1)),\nonumber
\end{align}
for all $i,j\in\omega$.

Equalities \eqref{eq-2.4} and \eqref{eq-2.6} imply that
\begin{align}\label{eq-2.7}
  (0,0,[1))\varepsilon^2 &=(0,0,[1))\big(\varepsilon{\upharpoonright}_{\boldsymbol{B}_{\omega}^{\mathscr{F}^4_{1,2,3}}}\circ \varepsilon{\upharpoonright}_{\boldsymbol{B}_{\omega}^{\mathscr{F}^4_{1,2,3}}}\big)= \nonumber \\
  &=(m,m,[3))\varepsilon{\upharpoonright}_{\boldsymbol{B}_{\omega}^{\mathscr{F}^4_{1,2,3}}}=\\
  &=(k(m+1)+m-1,k(m+1)+m-1,[1)). \nonumber
\end{align}
By equalities \eqref{eq-2.4}--\eqref{eq-2.6}  we have that $(\boldsymbol{B}_{\omega}^{\mathscr{F}^4_{p}})\varepsilon^2\subseteq \boldsymbol{B}_{\omega}^{\mathscr{F}^4_{p}}$ for any $p\in\{1,2,3\}$.  Since $(\boldsymbol{B}_{\omega}^{\mathscr{F}^4_{0}})\varepsilon^2\subseteq \boldsymbol{B}_{\omega}^{\mathscr{F}^4_{0}}$ the assumption of Proposition~\ref{proposition-2.1} holds for the monoid endomorphism $\varepsilon^2$, and hence we have that $(0,0,[1))\varepsilon^2=(0,0,[1))$. Then equality \eqref{eq-2.7} implies that $k(m+1)+m-1=0$. Since $m\in\omega$ and $k\geqslant 1$, we conclude that $m=0$ and $k=1$. By equalities \eqref{eq-2.4}--\eqref{eq-2.6} we get that
\begin{align*}
  (0,0,[1))\varepsilon&=(0,0,[3)); \\
  (0,0,[2))\varepsilon&=(0,0,[2)); \\
  (0,0,[3))\varepsilon&=(0,0,[1)),
\end{align*}
which imply that
\begin{align*}
  ((0,0,[1))\cdot(0,0,[3)))\varepsilon          &=(0,0,[3))\varepsilon=(0,0,[1)); \\
  (0,0,[1))\varepsilon\cdot(0,0,[3))\varepsilon &=(0,0,[3))\cdot(0,0,[1))=(0,0,[3)).
\end{align*}
The last two equalities contradict the assumption that $\varepsilon$ is an endomorphism of the semigroup  $\boldsymbol{B}_{\omega}^{\mathscr{F}^4}$. The obtained contradiction implies that the inclusion $(\boldsymbol{B}_{\omega}^{\mathscr{F}^4_{1}})\varepsilon{\upharpoonright}_{\boldsymbol{B}_{\omega}^{\mathscr{F}^4_{1,2,3}}}\subseteq \boldsymbol{B}_{\omega}^{\mathscr{F}^4_{1}}$ holds, and hence the statement of the proposition holds.
\end{proof}

\begin{proposition}\label{proposition-2.3}
There exists no an injective monoid endomorphism of the semigroup $\boldsymbol{B}_{\omega}^{\mathscr{F}^4}$ with the following properties: $(\boldsymbol{B}_{\omega}^{\mathscr{F}^4_{0,1,2}})\varepsilon\nsubseteq \boldsymbol{B}_{\omega}^{\mathscr{F}^4_{0,1,2}}$,
$(\boldsymbol{B}_{\omega}^{\mathscr{F}^4_{1,2,3}})\varepsilon\nsubseteq \boldsymbol{B}_{\omega}^{\mathscr{F}^4_{1,2,3}}$, and $(\boldsymbol{B}_{\omega}^{\mathscr{F}^4_{1}})\varepsilon\subseteq \boldsymbol{B}_{\omega}^{\mathscr{F}^4_{0}}$.
\end{proposition}

\begin{proof}
Suppose to the contrary that there exists an endomorphism $\varepsilon$ with the above described properties. Since for any $s\in\{0,1,2,3\}$ the subsemigroup $\boldsymbol{B}_{\omega}^{\mathscr{F}^4_{s}}$ of $\boldsymbol{B}_{\omega}^{\mathscr{F}^4}$ is isomorphic to the bicyclic semigroup, Proposition~4 of \cite{Gutik-Mykhalenych=2021} and the injectivity of $\varepsilon$ impliy that for any $p\in\{0,1,2,3\}$ there exists $q\in \{0,1,2,3\}$ such that $(\boldsymbol{B}_{\omega}^{\mathscr{F}^4_{p}})\varepsilon\subseteq \boldsymbol{B}_{\omega}^{\mathscr{F}^4_{q}}$. Then the inclusions $(\boldsymbol{B}_{\omega}^{\mathscr{F}^4_{0}})\varepsilon\subseteq \boldsymbol{B}_{\omega}^{\mathscr{F}^4_{0}}$ and $(\boldsymbol{B}_{\omega}^{\mathscr{F}^4_{1}})\varepsilon\subseteq \boldsymbol{B}_{\omega}^{\mathscr{F}^4_{0}}$ imply that only one of the following cases holds:
\begin{itemize}
  \item[$(i)$]   $(\boldsymbol{B}_{\omega}^{\mathscr{F}^4_{2}})\varepsilon\subseteq \boldsymbol{B}_{\omega}^{\mathscr{F}^4_{3}}$ and $(\boldsymbol{B}_{\omega}^{\mathscr{F}^4_{3}})\varepsilon\subseteq \boldsymbol{B}_{\omega}^{\mathscr{F}^4_{3}}$;
  \item[$(ii)$]  $(\boldsymbol{B}_{\omega}^{\mathscr{F}^4_{2}})\varepsilon\subseteq \boldsymbol{B}_{\omega}^{\mathscr{F}^4_{3}}$ and $(\boldsymbol{B}_{\omega}^{\mathscr{F}^4_{3}})\varepsilon\subseteq \boldsymbol{B}_{\omega}^{\mathscr{F}^4_{2}}$;
  \item[$(iii)$] $(\boldsymbol{B}_{\omega}^{\mathscr{F}^4_{2}})\varepsilon\subseteq \boldsymbol{B}_{\omega}^{\mathscr{F}^4_{3}}$ and $(\boldsymbol{B}_{\omega}^{\mathscr{F}^4_{3}})\varepsilon\subseteq \boldsymbol{B}_{\omega}^{\mathscr{F}^4_{1}}$;
  \item[$(iv)$]  $(\boldsymbol{B}_{\omega}^{\mathscr{F}^4_{2}})\varepsilon\subseteq \boldsymbol{B}_{\omega}^{\mathscr{F}^4_{3}}$ and $(\boldsymbol{B}_{\omega}^{\mathscr{F}^4_{3}})\varepsilon\subseteq \boldsymbol{B}_{\omega}^{\mathscr{F}^4_{0}}$.
\end{itemize}

Suppose that case $(i)$ holds. Since  the restriction $\varepsilon{\upharpoonright}_{\boldsymbol{B}_{\omega}^{\mathscr{F}^4_{0,1}}}\colon \boldsymbol{B}_{\omega}^{\mathscr{F}^4_{0,1}}\to \boldsymbol{B}_{\omega}^{\mathscr{F}^4_{0,1}}$ of the endomorphism $\varepsilon$ is an injective monoid endomorphism and by Proposition~1 of \cite{Gutik-Mykhalenych=2021} the semigroup $\boldsymbol{B}_{\omega}^{\mathscr{F}^2}$ is isomorphic to the subsemigroup $\boldsymbol{B}_{\omega}^{\mathscr{F}^4_{0,1}}$ of $\boldsymbol{B}_{\omega}^{\mathscr{F}^4}$, Theorem~1 of \cite{Gutik-Pozdniakova=2022} implies that there exist a positive integer $k$ and $p\in\{1,\ldots,k-1\}$ such that
\begin{equation}\label{eq-2.8}
  (i,j,[s))\varepsilon{\upharpoonright}_{\boldsymbol{B}_{\omega}^{\mathscr{F}^4_{0.1}}}=(i,j,[s))\varepsilon=
  \left\{
    \begin{array}{ll}
      (ki,kj,[0)),     & \hbox{if~} s=0;\\
      (ki+p,kj+p,[0)), & \hbox{if~} s=1.
    \end{array}
  \right.
\end{equation}
Again, since  the restriction $\varepsilon{\upharpoonright}_{\boldsymbol{B}_{\omega}^{\mathscr{F}^4_{2,3}}}\colon \boldsymbol{B}_{\omega}^{\mathscr{F}^4_{2,3}}\to \boldsymbol{B}_{\omega}^{\mathscr{F}^4_{2,3}}$ of the endomorphism $\varepsilon$ is an injective  endomorphism and by Proposition~1 of \cite{Gutik-Mykhalenych=2021} the semigroup $\boldsymbol{B}_{\omega}^{\mathscr{F}^2}$ is isomorphic to the subsemigroup $\boldsymbol{B}_{\omega}^{\mathscr{F}^4_{2,3}}$ of $\boldsymbol{B}_{\omega}^{\mathscr{F}^4}$, Theorem~2 of \cite{Gutik-Serivka=2025} implies that there exist a non-negative integer $n_1$, a positive integer $k_1$, and $p_1\in\{1,\ldots,k-1\}$ such that
\begin{equation}\label{eq-2.9}
  (i,j,[s))\varepsilon{\upharpoonright}_{\boldsymbol{B}_{\omega}^{\mathscr{F}^4_{2,3}}}=(i,j,[s))\varepsilon=
  \left\{
    \begin{array}{ll}
      (k_1i+p_1+n_1,k_1j+p_1+n_1,[3)), & \hbox{if~} s=2;\\
      (k_1i+n_1,k_1j+n_1,[3)),         & \hbox{if~} s=3.
    \end{array}
  \right.
\end{equation}

By Proposition~3 of \cite{Gutik-Mykhalenych=2021} we have that $(1,1,[1))\preccurlyeq(0,0,[2))$. Proposition~1.4.21 of \cite{Lawson=1998} implies that a homomorphism of inverse semigroups preserves the natural partial order, and hence by equalities \eqref{eq-2.8} and \eqref{eq-2.9} we get that
\begin{equation*}
  (k+p,k+p,[0))=(1,1,[1))\varepsilon\preccurlyeq(0,0,[2))\varpi=(p_1+n_1,p_1+n_1,[3)).
\end{equation*}
Then lemma~5 of \cite{Gutik-Mykhalenych=2020} implies that $k+p\geqslant p_1+n_1$ and $0\geqslant p_1+n_1+3$. Hence we have that  $k+p\geqslant p_1+n_1+3$. Then by equalities \eqref{eq-2.8} and \eqref{eq-2.9} we obtain that
\begin{align*}
  ((1,1,[1))\cdot(0,0,[2))\varepsilon&=(1,1,[2))\varepsilon= \\
   &=(k_1p_1+n_1,k_1p_1+n_1,[3))
\end{align*}
and
\begin{align*}
  (1,1,[1))\varepsilon\cdot(0,0,[2))\varepsilon&=(k+p,k+p,[0))\cdot(n_1,n_1,[3))= \\
   &=(k+p,k+p,[0)\cap (n_1-k-p+[3)))=\\
   &=(k+p,k+p,[0)),
\end{align*}
because $k+p\geqslant p_1+n_1+3$. These equalities contradict the assumption that $\varepsilon$ is an endomorphism.

We observe that $\varepsilon^2$ is an injective monoid endomorphism of the semigroup $\boldsymbol{B}_{\omega}^{\mathscr{F}^4}$ as a composition of injective monoid endomorphisms.

Suppose that case $(ii)$ holds. Then we have that $(\boldsymbol{B}_{\omega}^{\mathscr{F}^4_{0,1,2}})\varepsilon^2\subseteq \boldsymbol{B}_{\omega}^{\mathscr{F}^4_{0,2}} \subseteq \boldsymbol{B}_{\omega}^{\mathscr{F}^4_{0,1,2}}$ and hence by Proposition~\ref{proposition-2.1} we get that $\varepsilon^2$ is the identity transformation of the semigroup $\boldsymbol{B}_{\omega}^{\mathscr{F}^4}$. But we have that $(\boldsymbol{B}_{\omega}^{\mathscr{F}^4_{1}})\varepsilon^2\subseteq \boldsymbol{B}_{\omega}^{\mathscr{F}^4_{0}}$, a contradiction.

Suppose that case $(iii)$ holds. Then we have that $(\boldsymbol{B}_{\omega}^{\mathscr{F}^4_{0,1,2}})\varepsilon^2\subseteq \boldsymbol{B}_{\omega}^{\mathscr{F}^4_{0,1}} \subseteq \boldsymbol{B}_{\omega}^{\mathscr{F}^4_{0,1,2}}$ and hence by Proposition~\ref{proposition-2.1} the mapping  $\varepsilon^2$ is the identity transformation of the semigroup $\boldsymbol{B}_{\omega}^{\mathscr{F}^4}$. But we have that $(\boldsymbol{B}_{\omega}^{\mathscr{F}^4_{1}})\varepsilon^2\subseteq \boldsymbol{B}_{\omega}^{\mathscr{F}^4_{0}}$ and $(\boldsymbol{B}_{\omega}^{\mathscr{F}^4_{2}})\varepsilon^2\subseteq \boldsymbol{B}_{\omega}^{\mathscr{F}^4_{1}}$, a contradiction.

Suppose that case $(iv)$ holds. Again, since $(\boldsymbol{B}_{\omega}^{\mathscr{F}^4_{0,1,2}})\varepsilon^2\subseteq \boldsymbol{B}_{\omega}^{\mathscr{F}^4_{0}}\subseteq \boldsymbol{B}_{\omega}^{\mathscr{F}^4_{0,1,2}}$, Proposition~\ref{proposition-2.1} implies that the mapping  $\varepsilon^2$ is the identity transformation of the semigroup $\boldsymbol{B}_{\omega}^{\mathscr{F}^4}$. But we have that $(\boldsymbol{B}_{\omega}^{\mathscr{F}^4_{1}})\varepsilon^2\subseteq \boldsymbol{B}_{\omega}^{\mathscr{F}^4_{0}}$, $(\boldsymbol{B}_{\omega}^{\mathscr{F}^4_{2}})\varepsilon^2\subseteq \boldsymbol{B}_{\omega}^{\mathscr{F}^4_{0}}$, and $(\boldsymbol{B}_{\omega}^{\mathscr{F}^4_{3}})\varepsilon^2\subseteq \boldsymbol{B}_{\omega}^{\mathscr{F}^4_{0}}$, a contradiction.

The obtained contradictions imply the statement of the proposition.
\end{proof}

\begin{proposition}\label{proposition-2.4}
There exists no an injective monoid endomorphism of the semigroup $\boldsymbol{B}_{\omega}^{\mathscr{F}^4}$ with the following properties: $(\boldsymbol{B}_{\omega}^{\mathscr{F}^4_{0,1,2}})\varepsilon\nsubseteq \boldsymbol{B}_{\omega}^{\mathscr{F}^4_{0,1,2}}$,
$(\boldsymbol{B}_{\omega}^{\mathscr{F}^4_{1,2,3}})\varepsilon\nsubseteq \boldsymbol{B}_{\omega}^{\mathscr{F}^4_{1,2,3}}$, and $(\boldsymbol{B}_{\omega}^{\mathscr{F}^4_{1}})\varepsilon\subseteq \boldsymbol{B}_{\omega}^{\mathscr{F}^4_{1}}$.
\end{proposition}

\begin{proof}
Suppose to the contrary that there exists an injective monoid endomorphism $\varepsilon$ of $\boldsymbol{B}_{\omega}^{\mathscr{F}^4}$ which satisfies the assumptions of the proposition. The assumptions of the proposition imply that $(\boldsymbol{B}_{\omega}^{\mathscr{F}^4_{2}})\varepsilon\subseteq \boldsymbol{B}_{\omega}^{\mathscr{F}^4_{3}}$ and $(\boldsymbol{B}_{\omega}^{\mathscr{F}^4_{3}})\varepsilon\subseteq \boldsymbol{B}_{\omega}^{\mathscr{F}^4_{0}}$. Hence we get that
\begin{equation}\label{eq-2.10}
(\boldsymbol{B}_{\omega}^{\mathscr{F}^4_{1}})\varepsilon^2\subseteq \boldsymbol{B}_{\omega}^{\mathscr{F}^4_{1}}, \qquad (\boldsymbol{B}_{\omega}^{\mathscr{F}^4_{2}})\varepsilon^2\subseteq \boldsymbol{B}_{\omega}^{\mathscr{F}^4_{0}}, \qquad \hbox{and} \qquad (\boldsymbol{B}_{\omega}^{\mathscr{F}^4_{3}})\varepsilon^2\subseteq \boldsymbol{B}_{\omega}^{\mathscr{F}^4_{0}}.
\end{equation}
Since $\varepsilon^2$ is an injective monoid endomorphism of the semigroup $\boldsymbol{B}_{\omega}^{\mathscr{F}^4}$ as a composition injective monoid endomorphisms, the above inclusion imply that the assumptions of Proposition~\ref{proposition-2.1} holds for the map $\varepsilon^2$. Hence $\varepsilon^2$ is the identity transformation of the semigroup $\boldsymbol{B}_{\omega}^{\mathscr{F}^4}$ which contradicts inclusions \eqref{eq-2.10}. The obtained contradiction implies the statement of the proposition.
\end{proof}

\begin{proposition}\label{proposition-2.5}
There exists no an injective monoid endomorphism of the semigroup $\boldsymbol{B}_{\omega}^{\mathscr{F}^4}$ with the following properties: $(\boldsymbol{B}_{\omega}^{\mathscr{F}^4_{0,1,2}})\varepsilon\nsubseteq \boldsymbol{B}_{\omega}^{\mathscr{F}^4_{0,1,2}}$,
$(\boldsymbol{B}_{\omega}^{\mathscr{F}^4_{1,2,3}})\varepsilon\nsubseteq \boldsymbol{B}_{\omega}^{\mathscr{F}^4_{1,2,3}}$, and $(\boldsymbol{B}_{\omega}^{\mathscr{F}^4_{1}})\varepsilon\subseteq \boldsymbol{B}_{\omega}^{\mathscr{F}^4_{2}}$.
\end{proposition}

\begin{proof}
Suppose to the contrary that there exists an injective monoid endomorphism $\varepsilon$ of $\boldsymbol{B}_{\omega}^{\mathscr{F}^4}$ which satisfies the assumptions of the proposition. The assumptions of the proposition imply that $(\boldsymbol{B}_{\omega}^{\mathscr{F}^4_{2}})\varepsilon\subseteq \boldsymbol{B}_{\omega}^{\mathscr{F}^4_{3}}$ and $(\boldsymbol{B}_{\omega}^{\mathscr{F}^4_{3}})\varepsilon\subseteq \boldsymbol{B}_{\omega}^{\mathscr{F}^4_{0}}$. The mapping $\varepsilon^3$ is an injective monoid endomorphism of the semigroup $\boldsymbol{B}_{\omega}^{\mathscr{F}^4}$ as a composition injective monoid endomorphisms. Since $(\boldsymbol{B}_{\omega}^{\mathscr{F}^4_{0,1,2}})\varepsilon^3\subseteq \boldsymbol{B}_{\omega}^{\mathscr{F}^4_{0}} \subseteq \boldsymbol{B}_{\omega}^{\mathscr{F}^4_{0,1,2}}$, Proposition~\ref{proposition-2.1} implies that the mapping  $\varepsilon^3$ is the identity transformation of the semigroup $\boldsymbol{B}_{\omega}^{\mathscr{F}^4}$. But we have that $(\boldsymbol{B}_{\omega}^{\mathscr{F}^4_{1}})\varepsilon^3\subseteq \boldsymbol{B}_{\omega}^{\mathscr{F}^4_{0}}$, $(\boldsymbol{B}_{\omega}^{\mathscr{F}^4_{2}})\varepsilon^3\subseteq \boldsymbol{B}_{\omega}^{\mathscr{F}^4_{0}}$, and
$(\boldsymbol{B}_{\omega}^{\mathscr{F}^4_{3}})\varepsilon^3\subseteq \boldsymbol{B}_{\omega}^{\mathscr{F}^4_{0}}$, a contradiction. The obtained contradiction implies the statement of the proposition.
\end{proof}

\begin{proposition}\label{proposition-2.6}
There exists no an injective monoid endomorphism of the semigroup $\boldsymbol{B}_{\omega}^{\mathscr{F}^4}$ with the following properties: $(\boldsymbol{B}_{\omega}^{\mathscr{F}^4_{0,1,2}})\varepsilon\nsubseteq \boldsymbol{B}_{\omega}^{\mathscr{F}^4_{0,1,2}}$,
$(\boldsymbol{B}_{\omega}^{\mathscr{F}^4_{1,2,3}})\varepsilon\nsubseteq \boldsymbol{B}_{\omega}^{\mathscr{F}^4_{1,2,3}}$, $(\boldsymbol{B}_{\omega}^{\mathscr{F}^4_{1}})\varepsilon\subseteq \boldsymbol{B}_{\omega}^{\mathscr{F}^4_{3}}$ and $(\boldsymbol{B}_{\omega}^{\mathscr{F}^4_{2}})\varepsilon\subseteq \boldsymbol{B}_{\omega}^{\mathscr{F}^4_{0}}$.
\end{proposition}

\begin{proof}
Suppose to the contrary that there exists an injective monoid endomorphism $\varepsilon$ of $\boldsymbol{B}_{\omega}^{\mathscr{F}^4}$ which satisfies the assumptions of the proposition.Since for any $s\in\{0,1,2,3\}$ the subsemigroup $\boldsymbol{B}_{\omega}^{\mathscr{F}^4_{s}}$ of $\boldsymbol{B}_{\omega}^{\mathscr{F}^4}$ is isomorphic to the bicyclic semigroup, Proposition~4 of \cite{Gutik-Mykhalenych=2021} and the injectivity of $\varepsilon$ implies that for any $p\in\{0,1,2,3\}$ there exists $q\in \{0,1,2,3\}$ such that $(\boldsymbol{B}_{\omega}^{\mathscr{F}^4_{p}})\varepsilon\subseteq \boldsymbol{B}_{\omega}^{\mathscr{F}^4_{q}}$. Then the inclusions $(\boldsymbol{B}_{\omega}^{\mathscr{F}^4_{0}})\varepsilon\subseteq \boldsymbol{B}_{\omega}^{\mathscr{F}^4_{0}}$, $(\boldsymbol{B}_{\omega}^{\mathscr{F}^4_{1}})\varepsilon\subseteq \boldsymbol{B}_{\omega}^{\mathscr{F}^4_{3}}$, and $(\boldsymbol{B}_{\omega}^{\mathscr{F}^4_{2}})\varepsilon\subseteq \boldsymbol{B}_{\omega}^{\mathscr{F}^4_{0}}$ imply that only one of the following cases holds:
\begin{itemize}
  \item[$(i)$]   $(\boldsymbol{B}_{\omega}^{\mathscr{F}^4_{3}})\varepsilon\subseteq \boldsymbol{B}_{\omega}^{\mathscr{F}^4_{0}}$;
  \item[$(ii)$]  $(\boldsymbol{B}_{\omega}^{\mathscr{F}^4_{3}})\varepsilon\subseteq \boldsymbol{B}_{\omega}^{\mathscr{F}^4_{1}}$;
  \item[$(iii)$] $(\boldsymbol{B}_{\omega}^{\mathscr{F}^4_{3}})\varepsilon\subseteq \boldsymbol{B}_{\omega}^{\mathscr{F}^4_{2}}$;
  \item[$(iv)$]  $(\boldsymbol{B}_{\omega}^{\mathscr{F}^4_{3}})\varepsilon\subseteq \boldsymbol{B}_{\omega}^{\mathscr{F}^4_{3}}$.
\end{itemize}

We observe that the mappings $\varepsilon^2$ and $\varepsilon^3$ are injective monoid endomorphisms of the semigroup $\boldsymbol{B}_{\omega}^{\mathscr{F}^4}$ because they are compositions of injective monoid endomorphisms.

Suppose that case $(i)$ holds.  Since $(\boldsymbol{B}_{\omega}^{\mathscr{F}^4_{0,1,2}})\varepsilon^2\subseteq \boldsymbol{B}_{\omega}^{\mathscr{F}^4_{0}} \subseteq \boldsymbol{B}_{\omega}^{\mathscr{F}^4_{0,1,2}}$, Proposition~\ref{proposition-2.1} implies that the mapping  $\varepsilon^3$ is the identity transformation of the semigroup $\boldsymbol{B}_{\omega}^{\mathscr{F}^4}$.  But we have that  $(\boldsymbol{B}_{\omega}^{\mathscr{F}^4_{0,1,2}})\varepsilon^2\subseteq \boldsymbol{B}_{\omega}^{\mathscr{F}^4_{0}}$, a contradiction.

In cases $(ii)$ and $(iii)$ we have that $(\boldsymbol{B}_{\omega}^{\mathscr{F}^4_{0}})\varepsilon^2\subseteq \boldsymbol{B}_{\omega}^{\mathscr{F}^4_{0}}$, $(\boldsymbol{B}_{\omega}^{\mathscr{F}^4_{1}})\varepsilon^2\subseteq \boldsymbol{B}_{\omega}^{\mathscr{F}^4_{2}}$, and $(\boldsymbol{B}_{\omega}^{\mathscr{F}^4_{2}})\varepsilon^2\subseteq \boldsymbol{B}_{\omega}^{\mathscr{F}^4_{0}}$. This and  Proposition~\ref{proposition-2.1} imply that the mapping  $\varepsilon^2$ is the identity transformation of the semigroup $\boldsymbol{B}_{\omega}^{\mathscr{F}^4}$, a contradiction.

Suppose that case $(iv)$ holds. By the definition of the semigroup operation of $\boldsymbol{B}_{\omega}^{\mathscr{F}^4}$ (see formula \eqref{eq-1.2}) we have that
\begin{equation}\label{eq-2.11}
  (1,1,[0))\cdot(0,0,[2))=(1,1,[1)).
\end{equation}
Since a homomorphic image of an idempotent is an idempotent, too, we conclude that the assumptions of the proposition  imply that there exist non-negative integers $k_0$, $k_1$, and $k_2$ such that
\begin{align*}
  (1,1,[0))\varepsilon&=(k_0,k_0,[0)); \\
  (1,1,[1))\varepsilon&=(k_1,k_1,[3)); \\
  (0,0,[2))\varepsilon&=(k_2,k_2,[0)),
\end{align*}
and by equality \eqref{eq-2.11} we obtain that
\begin{equation*}
  ((1,1,[0))\cdot(0,0,[2)))\varepsilon=(1,1,[1))\varepsilon=(k_1,k_1,[3)).
\end{equation*}
But we have that
\begin{align*}
  (1,1,[0))\varepsilon\cdot(0,0,[2))\varepsilon&=(k_0,k_0,[0))\cdot(k_2,k_2,[0))= \\
   &=
   \left\{
     \begin{array}{ll}
       (k_2,k_2,(k_0-k_2+[0))\cap[0)), & \hbox{if~} k_0\leqslant k_2;\\
       (k_0,k_0,[0)\cap(k_2-k_0+[0))), & \hbox{if~} k_0> k_2
     \end{array}
   \right.\in\\
   &\in\boldsymbol{B}_{\omega}^{\mathscr{F}^4_{0}},
\end{align*}
a contradiction.

The obtained contradictions imply the statement of the proposition.
\end{proof}

\begin{proposition}\label{proposition-2.7}
There exists no an injective monoid endomorphism of the semigroup $\boldsymbol{B}_{\omega}^{\mathscr{F}^4}$ with the following properties: $(\boldsymbol{B}_{\omega}^{\mathscr{F}^4_{0,1,2}})\varepsilon\nsubseteq \boldsymbol{B}_{\omega}^{\mathscr{F}^4_{0,1,2}}$,
$(\boldsymbol{B}_{\omega}^{\mathscr{F}^4_{1,2,3}})\varepsilon\nsubseteq \boldsymbol{B}_{\omega}^{\mathscr{F}^4_{1,2,3}}$, $(\boldsymbol{B}_{\omega}^{\mathscr{F}^4_{1}})\varepsilon\subseteq \boldsymbol{B}_{\omega}^{\mathscr{F}^4_{3}}$, $(\boldsymbol{B}_{\omega}^{\mathscr{F}^4_{2}})\varepsilon\subseteq \boldsymbol{B}_{\omega}^{\mathscr{F}^4_{1}}$, and $(\boldsymbol{B}_{\omega}^{\mathscr{F}^4_{3}})\varepsilon\subseteq \boldsymbol{B}_{\omega}^{\mathscr{F}^4_{0}}$.
\end{proposition}

\begin{proof}
Simple verifications show that $(\boldsymbol{B}_{\omega}^{\mathscr{F}^4_{0}})\varepsilon^3\subseteq \boldsymbol{B}_{\omega}^{\mathscr{F}^4_{0}}$, $(\boldsymbol{B}_{\omega}^{\mathscr{F}^4_{1}})\varepsilon^3\subseteq \boldsymbol{B}_{\omega}^{\mathscr{F}^4_{0}}$, and $(\boldsymbol{B}_{\omega}^{\mathscr{F}^4_{2}})\varepsilon^3\subseteq \boldsymbol{B}_{\omega}^{\mathscr{F}^4_{0}}$. It is obvious that $\varepsilon^3$ is an injective monoid endomorphism of the semigroup $\boldsymbol{B}_{\omega}^{\mathscr{F}^4}$ as a composition of injective monoid endomorphisms. By Proposition~\ref{proposition-2.1}  the mapping  $\varepsilon^3$ is the identity transformation of the semigroup $\boldsymbol{B}_{\omega}^{\mathscr{F}^4}$, a contrdiction. The obtained contradiction implies the statement of the proposition.
\end{proof}

\begin{proposition}\label{proposition-2.8}
There exists no an injective monoid endomorphism of the semigroup $\boldsymbol{B}_{\omega}^{\mathscr{F}^4}$ with the following properties: $(\boldsymbol{B}_{\omega}^{\mathscr{F}^4_{0,1,2}})\varepsilon\nsubseteq \boldsymbol{B}_{\omega}^{\mathscr{F}^4_{0,1,2}}$,
$(\boldsymbol{B}_{\omega}^{\mathscr{F}^4_{1,2,3}})\varepsilon\nsubseteq \boldsymbol{B}_{\omega}^{\mathscr{F}^4_{1,2,3}}$, $(\boldsymbol{B}_{\omega}^{\mathscr{F}^4_{1}})\varepsilon\subseteq \boldsymbol{B}_{\omega}^{\mathscr{F}^4_{3}}$, $(\boldsymbol{B}_{\omega}^{\mathscr{F}^4_{2}})\varepsilon\subseteq \boldsymbol{B}_{\omega}^{\mathscr{F}^4_{2}}$, and $(\boldsymbol{B}_{\omega}^{\mathscr{F}^4_{3}})\varepsilon\subseteq \boldsymbol{B}_{\omega}^{\mathscr{F}^4_{0}}$.
\end{proposition}

\begin{proof}
We observe that that $\varepsilon^2$ is an injective monoid endomorphism of the semigroup $\boldsymbol{B}_{\omega}^{\mathscr{F}^4}$ as a composition of injective monoid endomorphisms such that $(\boldsymbol{B}_{\omega}^{\mathscr{F}^4_{0}})\varepsilon^2\subseteq \boldsymbol{B}_{\omega}^{\mathscr{F}^4_{0}}$, $(\boldsymbol{B}_{\omega}^{\mathscr{F}^4_{1}})\varepsilon^2\subseteq \boldsymbol{B}_{\omega}^{\mathscr{F}^4_{0}}$, and $(\boldsymbol{B}_{\omega}^{\mathscr{F}^4_{2}})\varepsilon^2\subseteq \boldsymbol{B}_{\omega}^{\mathscr{F}^4_{2}}$. Proposition~\ref{proposition-2.1} implies that the mapping  $\varepsilon^2$ is the identity transformation of the semigroup $\boldsymbol{B}_{\omega}^{\mathscr{F}^4}$, a contradiction. The obtained contradiction implies the statement of the proposition.
\end{proof}

\begin{proposition}\label{proposition-2.9}
There exists no an injective monoid endomorphism of the semigroup $\boldsymbol{B}_{\omega}^{\mathscr{F}^4}$ with the following properties: $(\boldsymbol{B}_{\omega}^{\mathscr{F}^4_{0,1,2}})\varepsilon\nsubseteq \boldsymbol{B}_{\omega}^{\mathscr{F}^4_{0,1,2}}$,
$(\boldsymbol{B}_{\omega}^{\mathscr{F}^4_{1,2,3}})\varepsilon\nsubseteq \boldsymbol{B}_{\omega}^{\mathscr{F}^4_{1,2,3}}$, $(\boldsymbol{B}_{\omega}^{\mathscr{F}^4_{1}})\varepsilon\subseteq \boldsymbol{B}_{\omega}^{\mathscr{F}^4_{3}}$, $(\boldsymbol{B}_{\omega}^{\mathscr{F}^4_{2}})\varepsilon\subseteq \boldsymbol{B}_{\omega}^{\mathscr{F}^4_{3}}$, and $(\boldsymbol{B}_{\omega}^{\mathscr{F}^4_{3}})\varepsilon\subseteq \boldsymbol{B}_{\omega}^{\mathscr{F}^4_{0}}$.
\end{proposition}

\begin{proof}
Since $\varepsilon^2$ is an injective monoid endomorphism of the semigroup $\boldsymbol{B}_{\omega}^{\mathscr{F}^4}$ such that $(\boldsymbol{B}_{\omega}^{\mathscr{F}^4_{0}})\varepsilon^2\subseteq \boldsymbol{B}_{\omega}^{\mathscr{F}^4_{0}}$, $(\boldsymbol{B}_{\omega}^{\mathscr{F}^4_{1}})\varepsilon^2\subseteq \boldsymbol{B}_{\omega}^{\mathscr{F}^4_{0}}$, and $(\boldsymbol{B}_{\omega}^{\mathscr{F}^4_{2}})\varepsilon^2\subseteq \boldsymbol{B}_{\omega}^{\mathscr{F}^4_{0}}$, Proposition~\ref{proposition-2.1} implies that the mapping  $\varepsilon^2$ is the identity transformation of the semigroup $\boldsymbol{B}_{\omega}^{\mathscr{F}^4}$. The obtained contradiction implies the statement of the proposition.
\end{proof}

We went through all the possible cases in Propositions~\ref{proposition-2.1}--\ref{proposition-2.9} and get the result which  summarize this in Theorem~\ref{theorem-2.10}.

\begin{theorem}\label{theorem-2.10}
Every injective monoid endomorphism of the semigroup $\boldsymbol{B}_{\omega}^{\mathscr{F}^4}$  is the identity transformation.
\end{theorem}

Theorem~\ref{theorem-2.10} implies Corollary~\ref{corollary-2.11} which generalizes Corollary~2 in \cite{Gutik-Mykhalenych=2021}.

\begin{corollary}\label{corollary-2.11}
The semigroup of all injective monoid endomorphisms of the semigroup $\boldsymbol{B}_{\omega}^{\mathscr{F}^4}$  is trivial.
\end{corollary}

Corollary~\ref{corollary-2.11} implies the following corollary.

\begin{corollary}[{Corollary~2 in \cite{Gutik-Mykhalenych=2021}}]\label{corollary-2.12}
The group of all automorphisms of the semigroup $\boldsymbol{B}_{\omega}^{\mathscr{F}^4}$  is trivial.
\end{corollary}

\section{\textbf{On injective monoid endomorphisms of the semigroup $\boldsymbol{B}_{\omega}^{\mathscr{F}^4}$}}\label{section-3}

By Proposition~1 of \cite{Gutik-Serivka=2026} the following mappings  $\lambda\colon \boldsymbol{B}_{\omega}^{\mathscr{F}^4}\to \boldsymbol{B}_{\omega}^{\mathscr{F}^4}$ and $\varpi_4\colon \boldsymbol{B}_{\omega}^{\mathscr{F}^4}\to \boldsymbol{B}_{\omega}^{\mathscr{F}^4}$ defined by the formulae
\begin{equation*}
  (i,j,[p))\lambda=(i+1,j+1,[p)) \quad \hbox{and} \quad (i,j,[p))\varpi_4=(i+p,j+p,[3-p)), \quad i,j\in\omega, p\in\{0,\ldots,3\},
\end{equation*}
are injective endomorphisms of the semigroup $\boldsymbol{B}_{\omega}^{\mathscr{F}^4}$. For  endomorphisms $\lambda$ and $\varpi_4$ the following equalities hold $\varpi_4^2=\lambda^3$ and $\varpi\lambda=\lambda\varpi$. Also, without loss of generality we may assume that $\lambda^0$ and $\varpi_4^0$ are the identity maps of $\boldsymbol{B}_{\omega}^{\mathscr{F}^4}$.

For any $p\in\{0,1,\ldots,3\}$ we define
\begin{equation*}
  \boldsymbol{B}_{\omega}^{\mathscr{F}^4}[(m,m,[p))]=(m,m,[p))\cdot \boldsymbol{B}_{\omega}^{\mathscr{F}^4}\cdot(m,m,[p))=(m,m,[p))\cdot \boldsymbol{B}_{\omega}^{\mathscr{F}^4}\cap \boldsymbol{B}_{\omega}^{\mathscr{F}^4}\cdot (m,m,[p)).
\end{equation*}

\begin{lemma}\label{lemma-3.1}
Let $\varepsilon$ be an injective endomorphism of the semigroup $\boldsymbol{B}_{\omega}^{\mathscr{F}^4}$. Then the followng statements hold:
\begin{itemize}
  \item[$(i)$]   $(0,0,[0))\varepsilon\notin\boldsymbol{B}_{\omega}^{\mathscr{F}^4_{p}}$, $p=1,2$;
  \item[$(ii)$]  $(\boldsymbol{B}_{\omega}^{\mathscr{F}^4})\varepsilon\nsubseteq\boldsymbol{B}_{\omega}^{\mathscr{F}^4_{0,1,2}}$;
  \item[$(iii)$] $(\boldsymbol{B}_{\omega}^{\mathscr{F}^4})\varepsilon\nsubseteq\boldsymbol{B}_{\omega}^{\mathscr{F}^4_{1,2,3}}$.
\end{itemize}
\end{lemma}

\begin{proof}
$(i)$ Suppose to the contrary that there exists an injective endomorphism $\varepsilon$ of $\boldsymbol{B}_{\omega}^{\mathscr{F}^4}$ such that $(0,0,[0))\varepsilon\in\boldsymbol{B}_{\omega}^{\mathscr{F}^4_{1}}$. Then the restriction $\varepsilon{\upharpoonright}_{\boldsymbol{B}_{\omega}^{\mathscr{F}^4_{0,1,2}}}\colon\boldsymbol{B}_{\omega}^{\mathscr{F}^4_{0,1,2}}\to \boldsymbol{B}_{\omega}^{\mathscr{F}^4_{0,1,2}}$ of $\varepsilon$ is an injective endomorphism of the subsemigroup $\boldsymbol{B}_{\omega}^{\mathscr{F}^4_{0,1,2}}$ of $\boldsymbol{B}_{\omega}^{\mathscr{F}^4}$. By Proposition~1 of \cite{Gutik-Mykhalenych=2021} the mapping $\mathfrak{I}\colon \boldsymbol{B}_{\omega}^{\mathscr{F}^3}\to \boldsymbol{B}_{\omega}^{\mathscr{F}^4_{0,1,2}}$, $(i,j,[p))\to (i,j,[p))$, $i,j\in\omega$, $p\in\{0,1,2\}$, is an isomorphism. Then the mapping $\varepsilon_{\mathfrak{I}}=\mathfrak{I}\circ\varepsilon\circ\mathfrak{I}^{-1}$ is an injective endomorphism of the semigroup $\boldsymbol{B}_{\omega}^{\mathscr{F}^3}$. Theorem~1 of \cite{Gutik-Serivka=2026} implies that there exist non-negative integers $k$ and $n$ and $m\in\{0,1\}$ such that $\varepsilon_{\mathfrak{I}}=\alpha_{[k]}\circ\lambda^n\circ\varpi^m$, where the mappings $\lambda$, $\alpha_{[k]}$ and $\varpi$ are defined by formulae \eqref{eq-2.1}, \eqref{eq-2.2} and \eqref{eq-2.3}, respectively. Hence we get that either $(0,0,[0))\varepsilon\in\boldsymbol{B}_{\omega}^{\mathscr{F}^4_{0}}$ or $(0,0,[0))\varepsilon\in\boldsymbol{B}_{\omega}^{\mathscr{F}^4_{2}}$, which contradicts the assumption $(0,0,[0))\varepsilon\in\boldsymbol{B}_{\omega}^{\mathscr{F}^4_{1}}$.

Suppose that there exists an injective endomorphism $\varepsilon$ of $\boldsymbol{B}_{\omega}^{\mathscr{F}^4}$ such that $(0,0,[0))\varepsilon\in\boldsymbol{B}_{\omega}^{\mathscr{F}^4_{2}}$. Since by Proposition~1$(ii)$ of \cite{Gutik-Serivka=2026} the mapping $\varpi_4$ is an injective endomorphism of the semigroup $\boldsymbol{B}_{\omega}^{\mathscr{F}^4}$, the composition $\varepsilon\circ\varpi_4$ is an injective endomorphism of $\boldsymbol{B}_{\omega}^{\mathscr{F}^4}$, as well. Simple verification shows that $(0,0,[0))(\varepsilon\circ \varpi_4)\in \boldsymbol{B}_{\omega}^{\mathscr{F}^4_{1}}$, which contradicts the previous part of the proof.

$(ii)$ Suppose the contrary: there exists an injective endomorphism $\varepsilon$ of $\boldsymbol{B}_{\omega}^{\mathscr{F}^4}$ such that  $(\boldsymbol{B}_{\omega}^{\mathscr{F}^4})\varepsilon\subseteq\boldsymbol{B}_{\omega}^{\mathscr{F}^4_{0,1,2}}$. Statement $(i)$ implies that $(0,0,[0))\varepsilon\in\boldsymbol{B}_{\omega}^{\mathscr{F}^4_{0}}$. By Lemma~2 of \cite{Gutik-Mykhalenych=2020} there exists a non-negative integer $m$ such that $(0,0,[0))\varepsilon=(m,m,[0))$, because a homomorphic image of an idempotent is an idempotent, too. By Lemma~2 of \cite{Gutik-Serivka=2026} the semigroups $\boldsymbol{B}_{\omega}^{\mathscr{F}^4}$ and $\boldsymbol{B}_{\omega}^{\mathscr{F}^4}[(m,m,[0))]$ are isomorphic by the mapping $\mathfrak{I}_m\colon \boldsymbol{B}_{\omega}^{\mathscr{F}^4}\to \boldsymbol{B}_{\omega}^{\mathscr{F}^4}[(m,m,[0))]$, $(i,j,[p))\mapsto (i+m,j+m,[p))$, $i,j\in\omega$, $p\in\{0,1,2,3\}$. Then the map $\varepsilon\circ\mathfrak{I}_m^{-1}\colon \boldsymbol{B}_{\omega}^{\mathscr{F}^4}\to \boldsymbol{B}_{\omega}^{\mathscr{F}^4}$ is an injective endomorphism of the semigroup $\boldsymbol{B}_{\omega}^{\mathscr{F}^4}$ as a composition of injective endomorphisms. Since $(0,0,[0))(\varepsilon\circ\mathfrak{I}_m^{-1})=(0,0,[0))$, the mapping $\varepsilon\circ\mathfrak{I}_m^{-1}$ is a monoid endomorphism of $\boldsymbol{B}_{\omega}^{\mathscr{F}^4}$. By Theorem~\ref{theorem-2.10} the endomorphism $\varepsilon\circ\mathfrak{I}_m^{-1}$ is the identity transformation of $\boldsymbol{B}_{\omega}^{\mathscr{F}^4}$. But the inclusion $(\boldsymbol{B}_{\omega}^{\mathscr{F}^4})\varepsilon\subseteq\boldsymbol{B}_{\omega}^{\mathscr{F}^4_{0,1,2}}$ implies that $(\boldsymbol{B}_{\omega}^{\mathscr{F}^4})(\varepsilon\circ\mathfrak{I}_m^{-1})\subseteq\boldsymbol{B}_{\omega}^{\mathscr{F}^4_{0,1,2}}$, a contradiction. The obtained contradiction implies statement $(ii)$.

$(iii)$ Suppose to the contrary that there exists an injective endomorphism $\varepsilon$ of $\boldsymbol{B}_{\omega}^{\mathscr{F}^4}$ such that $(\boldsymbol{B}_{\omega}^{\mathscr{F}^4})\varepsilon\subseteq\boldsymbol{B}_{\omega}^{\mathscr{F}^4_{1,2,3}}$. Since by Proposition~1 of \cite{Gutik-Serivka=2026} the mapping $\varpi_4\colon \boldsymbol{B}_{\omega}^{\mathscr{F}^4}\to \boldsymbol{B}_{\omega}^{\mathscr{F}^4}$ is an injective endomorphism of $\boldsymbol{B}_{\omega}^{\mathscr{F}^4}$, we conclude that the composition $\varepsilon\circ\varpi_4\colon \boldsymbol{B}_{\omega}^{\mathscr{F}^4}\to \boldsymbol{B}_{\omega}^{\mathscr{F}^4}$ is an injective endomorphism of $\boldsymbol{B}_{\omega}^{\mathscr{F}^4}$, too. Simple verifications show that $(\boldsymbol{B}_{\omega}^{\mathscr{F}^4})(\varepsilon\circ\varpi_4)\subseteq \boldsymbol{B}_{\omega}^{\mathscr{F}^4_{0,1,2}}$, which contradicts statement $(ii)$. The obtained contradiction implies that statement $(iii)$ holds.
\end{proof}

\begin{theorem}\label{theorem-3.2}
For every injective endomorphism $\varepsilon$ of the semigroup $\boldsymbol{B}_{\omega}^{\mathscr{F}^4}$  there exist a non-negative integer $k$ and $l\in\{0;1\}$ such that $\varepsilon=\lambda^k\varpi_4^l$.
\end{theorem}

\begin{proof}
Fix an arbitrary injective endomorphism $\varepsilon$ of the semigroup $\boldsymbol{B}_{\omega}^{\mathscr{F}^4}$.

If $(0,0,[0))\varepsilon=(0,0,[0))$, then $\varepsilon$ is a monoid endomorphism of $\boldsymbol{B}_{\omega}^{\mathscr{F}^4}$. Then by Theorem~\ref{theorem-2.10} the mapping $\varepsilon$ is the identity transformation of $\boldsymbol{B}_{\omega}^{\mathscr{F}^4}$, i.e., $\varepsilon=\lambda^0\varpi_4^0$.

Suppose that $(0,0,[0))\varepsilon\neq(0,0,[0))$. By Lemma~\ref{lemma-3.1} we have that either $(0,0,[0))\varepsilon\in\boldsymbol{B}_{\omega}^{\mathscr{F}^4_{0}}$ or $(0,0,[0))\varepsilon\in\boldsymbol{B}_{\omega}^{\mathscr{F}^4_{3}}$. If $(0,0,[0))\varepsilon\in\boldsymbol{B}_{\omega}^{\mathscr{F}^4_{0}}\setminus\{(0,0,[0))\}$, then by Lemma~2 of \cite{Gutik-Mykhalenych=2020} there exists a positive integer $k$ such that $(0,0,[0))\varepsilon=(k,k,[0))$, because a homomorphic image of an idempotent is an idempotent, too. Since $(0,0,[0))$ is the unit of the monoid $\boldsymbol{B}_{\omega}^{\mathscr{F}^4}$, we conclude that $(\boldsymbol{B}_{\omega}^{\mathscr{F}^4})\varepsilon\subseteq \boldsymbol{B}_{\omega}^{\mathscr{F}^4}[(k,k,[0))]$. By Corollary~1 of \cite{Gutik-Serivka=2026} the semigroups $\boldsymbol{B}_{\omega}^{\mathscr{F}^4}$ and $\boldsymbol{B}_{\omega}^{\mathscr{F}^4}[(k,k,[0))]$ are isomorphic by the mapping $\mathfrak{I}_k\colon \boldsymbol{B}_{\omega}^{\mathscr{F}^4}\to \boldsymbol{B}_{\omega}^{\mathscr{F}^4}[(k,k,[0))]$, $(i,j,[p))\mapsto(i+k,j+k,[p))$, $i,j\in\omega$, $p\in\{0,1,2,3\}$. It is obvious that the composition  $\varepsilon\circ \mathfrak{I}_k^{-1}$ is an injective endomorphism of the semigroup $\boldsymbol{B}_{\omega}^{\mathscr{F}^4}$ such that $(0,0,[0))(\varepsilon\circ \mathfrak{I}_k^{-1})=(0,0,[0))$. By Theorem~\ref{theorem-2.10}, $\varepsilon\circ \mathfrak{I}_k^{-1}$ is the identity transformation of $\boldsymbol{B}_{\omega}^{\mathscr{F}^4}$, and hence $\varepsilon=\mathfrak{I}_k$. Since $\mathfrak{I}_k=\lambda^k$, we get that $\varepsilon=\lambda^k$.

If $(0,0,[0))\varepsilon\in\boldsymbol{B}_{\omega}^{\mathscr{F}^4_{3}}$, then by Lemma~2 of \cite{Gutik-Mykhalenych=2020} there exists a non-negative integer $k$ such that $(0,0,[0))\varepsilon=(k,k,[3))$, because a homomorphic image of an idempotent is an idempotent, as well. Put $\varepsilon^{\prime}=\varepsilon\circ\varpi_4$. By Proposition~1 of \cite{Gutik-Serivka=2026} we get that
\begin{align*}
  (0,0,[0))\varepsilon^{\prime}&=(0,0,[0))(\varepsilon\circ\varpi_4)= \\
   &=(k,k,[3))\varpi_4= \\
   &=(k+3,k+3,[0)).
\end{align*}
By the above part of the proof we obtain that $\varepsilon^{\prime}=\lambda^{k+3}$, i.e., $\varepsilon\circ\varpi_4=\lambda^{k+3}$. Hence we get that $\varepsilon\circ\varpi_4^2=\lambda^{k+3}\circ\varpi_4$. Proposition~1 of \cite{Gutik-Serivka=2026} implies that $\varpi_4^2=\lambda^{3}$, and hence we have that
\begin{equation*}
  \varepsilon\circ\lambda^3=\lambda^{k+3}\circ\varpi_4=\varpi_4\circ\lambda^{k+3}.
\end{equation*}
Since $\varepsilon$, $\lambda$ and $\varpi_4$ are injective transformations of $\boldsymbol{B}_{\omega}^{\mathscr{F}^4}$, we conclude that $\varepsilon=\lambda^{k}\circ\varpi_4=\varpi_4\circ\lambda^{k}$.

This completes the proof of the theorem.
\end{proof}




\end{document}